\documentclass[12pt,sort&compress]{elsarticle}

\usepackage{amsmath}
\usepackage{amsfonts}
\usepackage{amssymb}
\usepackage{amsfonts}
\usepackage{amsthm}
\usepackage{amscd}
\usepackage{mathtools}
\usepackage{commath}
\usepackage{lmodern}
\usepackage{color}
\usepackage{a4wide}
\usepackage{bm} 
\usepackage{todonotes}
\usepackage{hyperref}

\usepackage[czech,english]{babel}
\usepackage[IL2]{fontenc}
\usepackage[utf8]{inputenc}

\newtheorem{theorem}{Theorem}[section]
\newtheorem{lemma}[theorem]{Lemma}

\theoremstyle{definition}
\newtheorem{definition}[theorem]{Definition}
\theoremstyle{remark}
\newtheorem{remark}[theorem]{\upshape\bfseries Remark}
\newtheorem{example}[theorem]{\upshape\bfseries Example}

\newcommand{\C}{\mathbb{C}}
\newcommand{\R}{\mathbb{R}}
\renewcommand{\H}{\mathbb{H}}
\renewcommand{\DH}{\mathbb{DH}}
\newcommand{\eps}{\varepsilon}
\newcommand{\SE}[1][3]{\operatorname{SE}(#1)}
\newcommand*{\qi}{\mathbf{i}}
\newcommand*{\qj}{\mathbf{j}}
\newcommand*{\qk}{\mathbf{k}}
\newcommand*{\ci}{\mathrm{i}}
\newcommand*{\Cj}[1]{{#1}^\ast}
\renewcommand{\vec}[1]{\mathbf{#1}}

\makeatletter
\newcommand{\eqnum}{\refstepcounter{equation}\textup{\tagform@{\theequation}}}
\makeatother

\DeclareMathOperator{\rank}{rank}
\DeclareMathOperator{\Span}{span}

\begin{document}

\begin{frontmatter}
\journal{Computer Aided Geometric Design}

\title{Rational Framing Motions and\\Spatial Rational Pythagorean Hodograph Curves}

\author[1]{Bahar Kalkan}
\ead{baharkalkan@yyu.edu.tr}
\address[1]{Van Yuzuncu Yil University, Department of Mathematics, Van, Turkey, 65090}

\author[2]{Daniel F. Scharler\fnref{fn1}}
\ead{daniel.scharler@uni-graz.at}
\address[2]{Universität Graz, Institute of Mathematics and Scientific Computing,
  Heinrichstr.~4, 8010 Graz, Austria}
\fntext[fn1]{Daniel F. Scharler contributed significantly to the originally
  submitted version of this manuscript. He tragically passed away in April 2022
  at the age of only 29.}

\author[3]{Hans-Peter Schröcker}
\ead{hans-peter.schroecker@uibk.ac.at}
\address[3]{Universität Innsbruck, Department of Basic Sciences in Engineering Sciences, Technikerstr.~13, 6020 Innsbruck, Austria}

\author[4]{Zbyněk Šír}
\ead{zbynek.sir@karlin.mff.cuni.cz}
\address[4]{Charles University, Faculty of Mathematics and Physic, Sokolovsk\'a 83, Prague 186 75, Czech Republic}

\begin{abstract}
  We propose a new method for constructing rational spatial Pythagorean Hodograph (PH) curves based on determining a suitable rational framing motion. While the spherical component of the framing motion is arbitrary, the translation part is determined be a modestly sized and nicely structured system of linear equations. Rather surprisingly, generic input data will only result in polynomial PH curves. We provide a complete characterization of all cases that admit truly rational (non-polynomial) solutions. Examples illustrate our ideas and relate them to existing literature.
\end{abstract}


\end{frontmatter}

\section{Introduction}
\label{sec:intro}

The distinctive property of a polynomial Pythagorean-hodograph (PH) curve is that its parametric speed, which specifies the rate of change of arc length with respect to the curve parameter, is a polynomial rather than the square root of a polynomial. This feature endows PH curves with many computational advantages in geometric design, motion control, animation, path planning, and similar applications.

Planar PH curves were first introduced in \cite{farouki90c}, through direct integration of hodograph components expressed in terms of real preimage polynomials. Subsequently, a complex-variable model for planar PH curves was introduced in \cite{farouki94}, that greatly facilitates the development of algorithms for their construction and analysis, see e.g. \cite{farouki08} for an extensive bibliography.

Spatial PH curves were first considered in \cite{farouki94a} using a three polynomial preimage. A characterization of Pythagorean polynomial quadruples in terms of four polynomials was presented, in a different context, in \cite{Dietz}. Subsequently, two algebraic models for spatial PH curves were based on this characterization: The quaternion and Hopf map representations, as proposed in \cite{choi02b}. These forms are rotation invariant \cite{farouki02a} and serve as the foundation for many practical constructions of spatial PH curves and associated frames \cite{farouki08,faroukisurv}.

The extension from polynomial to rational PH curves is non-trivial, since a rational hodograph does not always yield a rational curve upon integration. To circumvent this difficulty, a different approach was adopted in \cite{fiorot94,Pottmann} based on the dual (line) representation of planar curves. Specifically, plane curves are viewed as the envelopes of one-parameter families of tangent lines, rather than point loci. By assigning rational functions to describe the orientation of these tangents and their distance from the origin, one can construct rational PH curves in a geometrically intuitive manner \cite{pottmann95b}. A comparison of polynomial and rational planar PH curves can be found in \cite{farouki96c}.

The generalization of rational PH curves from the planar case to the spatial case is also non-trivial, since the dual line representation cannot be simply carried over: In the plane, points and lines are dual elements, but in space the duality is between points and planes, while lines are self-dual elements. This problem was solved in \cite{FaroukiSir} and a construction in two steps was given. In the first step a rational tangent indicatrix of the curve is constructed via stereographic projection. In the second step the osculating planes are obtained with one degree of freedom depending rationally on the parameter. Exploiting the property that every spatial curve has an associated tangent developable surface, the PH curve is determined from its tangent developable as the singular locus, or edge of regression. In \cite{krajnc1,FaroukiSir2} the first step was modified and the tangent indicatrix was described by a quaternion valued polynomial. In this way a comparatively simpler formula was provided (c.f. Equation~\eqref{eq:tr2} below). While this formula is quite elegant and provides a full description of all rational PH curves, it is difficult to have an insight in the possible (and frequent) cancellations between the numerator and denominator.

In \cite{FaroukiSir} the theory of rational PH curves was also connected to the theory of polynomial ones, in particular concerning the construction of rotation minimizing frames. Other results on rational spatial PH curves include \cite{barton} where spherical rational curves with rational rotation minimizing frame are constructed by applying Möbius transformations in $\R^3$ to piecewise planar PH cubics. In the series of papers \cite{krajnc1,krajnc2,krajnc3} the authors exploit the dual representation to interpolate with rational spatial PH curves of low class. In \cite{FaroukiSakkalis2019} a special form of the rational hodograph is used to construct planar rational PH curves with rational arc-length function. A generalization of this idea to the spatial curves is also suggested and shown on one example.

Our novel approach to rational PH curve starts with the observation that their rational tangent indicatrix gives rise to a rational spherical motion. When composed with the translation along the curve, it is a framing motion of the curve. The spherical component has been called Euler-Rodrigues motion or, equivalently referring to the images of a suitable orthogonal tripod, Euler-Rodrigues frame \cite{choi02,krajnc1}.

Writing the framing motion of the yet undetermined rational PH curve in the dual quaternion model of space kinematics, c.f. \cite{husty12} or \cite[Chapter~11]{selig05}, separating spherical and translational motion component, and selecting appropriate parameters (the spherical motion component, represented by a quaternion polynomial $A \in \H[t]$ and the PH curve's denominator polynomial $\alpha$) the PH conditions can be expressed by a system of linear equations.

By design, this approach will produce all rational PH curves of a certain maximal degree. In particular, it comprises the usual construction of polynomial PH curves from $A$ via integration. Rather surprisingly, for a generic choice of $A$ and $\alpha$ only polynomial solutions can be obtained. A detailed analysis of the system of linear equations yields a complete characterization of those cases that admit truly rational (non-polynomial) solutions (Theorem~\ref{th:non-polynomial-solutions}) in terms of the coefficients of the Taylor polynomial of the non-normalized tangent indicatrix at a zero of $\alpha$ and the multiplicity $n$ of this zero.

Our approach not only provides a straightforward and direct method to compute rational PH curves and also polynomial PH curves from the same data. It also sheds new light on rational PH curves and on their relation to polynomial PH curves. Via $A$ and $\alpha$, essential properties of the resulting family of rational PH curves can be controlled. This and the vector space structure of the solution family is assumed to be beneficial for many applications.

We continue this article by providing some background on the construction of
rational and polynomial PH curves and on the dual quaternion model of space
kinematics in Section~\ref{sec:preliminaries}. In
Section~\ref{sec:rational-PH-curves} we describe in detail how to compute
rational PH curves from the spherical component of a rational framing motion and
establish some basic properties of the construction. Existence of non-polynomial
PH curves to given input data is discussed in detail in
Section~\ref{sec:existence}.
In Section~\ref{sec:examples} we demonstrate how to compute some new and some
known examples of rational PH curves and we relate computation methods from
literature to our approach.

\section{Preliminaries}
\label{sec:preliminaries}

In this section we recall basic definitions and known results about PH curves, quaternions, dual quaternions and their relation to space kinematics. Let us start with the definition of PH curves. We suggest the analogous wording for both polynomial and rational cases postponing the necessary differences to remarks below.

\begin{definition}
  \label{def:PH-curve}
	A spatial polynomial (rational) parametric curve $\vec{r}(t)=(r_1(t),r_2(t),r_3(t))$ is called \emph{Pythagorean hodograph (PH)} if it has a polynomial (rational) speed. More precisely if there exists a polynomial (rational) function $\sigma(t)$ so that
  \begin{equation}
    \label{def}
    r_1'(t)^2+r_2'(t)^2+r_3'(t)^2=\sigma(t)^2.
  \end{equation}
\end{definition}

Note that there is an important difference between the polynomial and the rational case. Unlike the polynomial PH curves, the rational PH curves typically do not possess a rational arc-length function. Indeed, the primitive function of the speed $\int \sigma(t){\rm d}t$ is rational only in special cases \cite{FaroukiSakkalis2019}. This might be one of the reasons why in the plane rational PH curves are often defined as the curves having rational offsets \cite{Pottmann}. We will suggest later that for space curves the analogous geometric property might be the existence of a related framing motion.

Let us also stress that the property to be PH may depend on the parameterization. As a well known example consider the parabola $\vec{r}(s)=(s, s^2)$ which is not PH while its rational reparameterization $\vec{r}(t)=(\frac{t}{1-t^2},\frac{t^2}{(1-t^2)^2})$ is PH with speed function $\sigma(t)=\frac{(1+t^2)^2}{(1-t^2)^3}$. Therefore we consider the PH property as an aspect of the particular parameterization rather than of the geometric object. For this reason we will not use non-linear reparameterizations.

Equation \eqref{def} implies existence of the unit vector field
\begin{equation*}
  \vec{T}(t)=\biggl( \frac{r'_1(t)}{\sigma(t)}, \frac{r'_2(t)}{\sigma(t)}, \frac{r'_3(t)}{\sigma(t)} \biggr)
\end{equation*}
tangent to the curve. $\vec{T}(t)$ is defined for all $t\in \R$, possibly with exception of finitely many values. It can be seen as a rational spherical curve called \emph{tangent indicatrix}. In fact, existence of a \emph{rational} tangent indicatrix is equivalent to the PH condition for rational curves.

Next, we introduce the algebra of dual quaternions and the concept of representing motions via certain dual quaternion polynomials, c.f. \cite{husty12} or \cite[Chapter~11]{selig05}. The skew field of quaternions $\H$ is the associative real algebra generated by the four basis elements $1$, $\qi$, $\qj$, $\qk$ together with the multiplication derived from the generating relations
\begin{equation*}
  \qi^2 = \qj^2 = \qk^2 = \qi\qj\qk = -1.
\end{equation*}
\emph{Conjugation} of a quaternion $\mathcal Q = q_0 + q_1 \qi + q_2 \qj + q_3 \qk \in
\H$ is defined by changing signs of the coefficients at the complex units, i.e.
$\Cj{\mathcal Q} = q_0 - q_1 \qi - q_2 \qj - q_3 \qk$, and the norm of $q$ is given by
$\mathcal Q\Cj{\mathcal Q} = q_0^2 + q_1^2 + q_2^2 + q_3^2 \in \R$. The algebra $\DH$ of dual
quaternions is obtained by adjoining an element $\eps$ which commutes with
quaternions and has the property $\eps^2 = 0$. A dual quaternion can be written
as $\mathcal H = \mathcal P + \eps \mathcal D$ where \emph{primal part $\mathcal P$} and \emph{dual part $\mathcal D$} are
quaternions. The dual quaternion conjugate of $\mathcal H$ is defined as $\Cj{\mathcal H} = \Cj{\mathcal P}
+ \eps \Cj{\mathcal D}$ and its norm is $\mathcal H\Cj{\mathcal H} = \mathcal P\Cj{\mathcal P} + \eps(\mathcal P\Cj{\mathcal D} + \mathcal D\Cj{\mathcal P})$.
Note that the norm is a dual number in general and a real number if the
\emph{Study condition} $\mathcal P\Cj{\mathcal D} + \mathcal D\Cj{\mathcal P} = 0$ is fulfilled.
In this case (and assuming $\mathcal P \neq 0$), the multiplicative inverse of $\mathcal H$ is given by $\mathcal H^{-1} = \Cj{\mathcal H}/(\mathcal P\Cj{\mathcal P})$.

Denote by $\DH^\times$ the group of invertible dual quaternions $\mathcal H = \mathcal P + \eps \mathcal D$ that satisfy the \emph{Study condition $\mathcal P\Cj{\mathcal D} + \mathcal D\Cj{\mathcal P} = 0$.} The factor group $\DH^\times/\R^\times$ of $\DH^\times$ modulo the real multiplicative group $\R^\times$ is isomorphic to the group $\SE$ of rigid body displacements. The dual quaternion $\mathcal H = \mathcal P + \eps \mathcal D \in \DH^\times$ acts on the point with Cartesian coordinates $(x_1,x_2,x_3)$, embedded into $\DH$ as $\vec{x} = 1 + \eps(x_1\qi + x_2\qj + x_3\qk)$, via
\begin{equation}
  \label{eq:dq-action}
  \vec{x} \mapsto
  \frac{(\mathcal P - \eps \mathcal D) \vec{x} (\Cj{\mathcal P} + \eps \Cj{\mathcal D})}{\mathcal P\Cj{\mathcal P}}.
\end{equation}

Equation~\eqref{eq:dq-action} describes a single rigid-body displacement in dual quaternions. In order to obtain a motion, primal part $\mathcal P$ and dual part $\mathcal D$ (and hence also $\mathcal H = \mathcal P + \eps \mathcal D$) should depend on a real motion parameter $t$. Since we are interested in rational motions, we replace $\mathcal H$ by a \emph{dual quaternion polynomial $\mathcal C(t) = \mathcal A(t) + \eps \mathcal B(t) \in \DH[t]$} in the real variable $t$ that satisfies the polynomial Study condition $\mathcal A(t)\Cj{\mathcal B}(t)+\mathcal B(t)\Cj{\mathcal A}(t) = 0$ and $\mathcal A(t) \neq 0$. Polynomials of this type are called \emph{motion polynomials} \cite{hegedus13}. Clearly, the trajectory
\begin{equation}
  \label{eq:trajectory}
  \frac{(\mathcal A(t) - \eps \mathcal B(t))x(\Cj{\mathcal A}(t) + \eps \Cj{\mathcal B}(t))}{\mathcal A(t)\Cj{\mathcal A}(t)}
\end{equation}
of a point $x$ is a rational curve and it is well-known that all \emph{rational motions} (that is, motions where all trajectories are rational curves) can be generated in that way~\cite{juettler93}.

The action \eqref{eq:dq-action} of dual quaternions (and also the action of motion polynomials) on points can be extended to vectors. Their images depend only on the spherical component $\mathcal A(t)$ of the motion $\mathcal C(t) = \mathcal A(t) + \eps \mathcal B(t)$. Embedding the vector $x = (x_1,x_2,x_3)$ into $\DH$ as the \emph{vectorial} quaternion $\vec x = x_1\qi + x_2\qj + x_3\qk$, the image of $x$ is given by $\mathcal A(t)\vec x\Cj{\mathcal A}(t)/(\mathcal A(t)\Cj{\mathcal A}(t))$.

Quaternion polynomials have been proven to be very useful in order to
systematically construct polynomial PH curves. All spatial polynomial PH curves can be obtained by taking an arbitrary quaternion valued polynomial $\mathcal A(t)$ and a real polynomial $\lambda(t)$ and defining
\begin{equation}
  \label{racint}
  \vec{r}(t)=\int \lambda(t)\mathcal A(t) \qi \Cj{\mathcal A}(t)\,\mathrm{d}t,
\end{equation}
\cite{farouki08}. In principle all rational PH curves can be obtained in the same way by allowing $\lambda(t)$ to be rational \cite{FaroukiSir,FaroukiSir2, FaroukiSakkalis2019}. The integral \eqref{racint} however does not need to produce a rational curve and only a (linear) subset of rational functions $\lambda(t)$ may be used. A complete characterization of this subset seems to be a very difficult problem. Note, however, that the linear occurrence of $\lambda(t)$ in \eqref{racint} indicates that, given two rational PH curves $\tilde{\vec{r}}(t)$ and $\hat{\vec{r}}(t)$ with the same indicatrix $\vec{T}(t)$, their sum $\vec{r}(t)=\tilde{\vec{r}}(t)+\hat{\vec{r}}(t)$ has the same indicatrix and is therefore PH as well. Also for any constant real $c \in \R$ and any constant vector $\vec{d} \in \R^3$ the scaled and translated curve $c \vec{r}(t) + \vec{d}$ has the same indicatrix.

An explicit construction of a general rational PH curve bypassing the integration issue was given in \cite{FaroukiSir,krajnc1,FaroukiSir2}. Consider any quaternion valued polynomial $\mathcal A(t)$ and define the vector fields $\vec{v}(t)=\mathcal A(t)\qi \Cj{\mathcal A}(t)$ and $\vec{u}(t)=\vec{v}(t)\times\vec{v}'(t)$. While $\vec{v}(t)$ points in the tangent direction of the curve (it is a multiple of $\vec{T}(t)$) the field $\vec{u}(t)$ is perpendicular to the osculation plane (it is a multiple of the binormal vector). Then the osculation plane will have the implicit equation $\vec{u}(t)\cdot(x_1,x_2,x_3)=f(t)$ where $f(t)$ is an arbitrary rational function. The curve parameterization in terms of $\vec{u}(t)$ and $f(t)$ has the form
\begin{equation}
  \label{eq:tr2}
  \vec{r}(t)=\frac{f(t)   \, \vec{u}'(t)\times \vec{u}''(t) +
    f'(t)  \, \vec{u}''(t)\times \vec{u}(t) +
    f''(t) \, \vec{u}(t)\times \vec{u}'(t)}
  {\det[\,\vec{u}(t),\vec{u}'(t),\vec{u}''(t)\,]}.
\end{equation}
Note that for a fixed $A(t)$ the correspondence between $f(t)$ and $\vec{r}(t)$ is linear. While this formula is quite elegant and provides a full description of all the rational PH curves, it is difficult to have an insight in the possible cancellations between the numerator and denominator. This is due to the fact that the final denominator of $\vec{r}(t)$ comes both from $\det[\,\vec{u}(t),\vec{u}'(t),\vec{u}''(t)\,]$ and from the denominator of $f(t)$ and its derivatives. In particular it is difficult to understand for which $f(t)$ polynomial PH curves occur (as subset of the rational cases).

\section{Rational PH Curves and Framing Motions}
\label{sec:rational-PH-curves}

We are going to study PH curves as trajectories of rational framing motions and we will use the formalism of dual quaternions introduced in the previous section for that purpose. For simplicity of notation we will usually omit the indeterminate $t$ from now on, e.g. we write $\vec{r} = \vec{r}(t)$ or $\mathcal A = \mathcal A(t)$.

\begin{lemma}
  \label{lem:framing-motion}
  Given a rational parametric curve
  \begin{equation*}
    \vec{r} =
    \frac{1}{\alpha}
    (r_1, r_2, r_3),
    \quad
    \alpha, r_1, r_2, r_3 \in \R[t],
  \end{equation*}
  any rational motion such that the trajectory of the origin equals $\vec{r}$ is of the form
  \begin{equation}
    \label{eq:draw-curve}
    \mathcal C = (\alpha + \eps \vec b)\mathcal A
  \end{equation}
  where $\mathcal A \in \H[t] \setminus \{0\}$ is arbitrary and $\vec b \in \H[t]$ is of the shape $\vec b = -\frac{1}{2}(r_1\qi + r_2\qj + r_3\qk)$.
\end{lemma}

\begin{proof}
  By \eqref{eq:trajectory}, the translation along $\vec{r}$ is given by the motion polynomial $\alpha + \eps \vec b$ where $\vec b$ is defined as in the Lemma's statement. For $x = 1$ and $\vec b = 0$, \eqref{eq:trajectory} boils down to $1$. Since the origin is embedded into $\DH$ as $1$, we see that $A$ fixes the origin and the motion polynomial $\mathcal C$ as defined above indeed moves it along $\vec{r} = -2\alpha^{-1}\vec b$.

  Conversely, consider a motion polynomial $\mathcal C = \mathcal P + \eps \mathcal D \in \DH[t]$ and
  assume that $\vec{r} = \alpha^{-1}(r_1,r_2,r_3)$ is the trajectory of the
  origin. The trajectory of the origin under the motion $\mathcal C$ is
  $(\mathcal P\Cj{\mathcal D}-\mathcal D\Cj{\mathcal P})/(\mathcal P\Cj{\mathcal P})$ which, by the Study condition $\mathcal P\Cj{\mathcal D} + \mathcal D\Cj{\mathcal P}
  = 0$, equals $-2\mathcal D\Cj{\mathcal P}/(\mathcal P\Cj{\mathcal P})$. From this we infer that $\alpha$ divides
  $\mathcal P\Cj{\mathcal P}$, i.e., $\mathcal P\Cj{\mathcal P} = \lambda\alpha$ with $\lambda \in \R[t]$, whence
  $\lambda \vec b = \mathcal D\Cj{\mathcal P}$. Defining $\mathcal A \coloneqq \mathcal P$ we have
  \begin{equation*}
    (\alpha + \eps \vec b)\mathcal A
    = \alpha \mathcal P + \eps \vec b\mathcal P
    = \alpha \mathcal P + \eps \tfrac{1}{\lambda}D\underbrace{\Cj{\mathcal P}\mathcal P}_{=\lambda\alpha}
    = \alpha(\mathcal P + \eps \mathcal D)
    = \alpha \mathcal C
  \end{equation*}
  so that the two motion polynomials $\mathcal C$ and $(\alpha + \eps \vec b)\mathcal A$ are indeed equal up to a real polynomial factor and therefore represent the same motion.
\end{proof}

Note that we can assume that $\mathcal A$ is reduced, that is, it does not have a real polynomial factor of positive degree. If common real factors of $\alpha$ and $\vec b\mathcal A$ appear, they can be divided off from $\mathcal C$ in order to obtain a motion polynomial of lower degree which looks different from \eqref{eq:draw-curve}. Nonetheless, \eqref{eq:draw-curve} is the general way of writing rational motions to draw the parametric rational curve $\vec{r}$ as trajectory of the origin. We will use precisely this form to describe all the rational PH curves.

\begin{lemma}
  \label{lem:vectorial}
  The polynomial $\mathcal C = (\alpha + \eps \mathcal B)\mathcal A$ with $\alpha \in \R[t]$ and $\mathcal A$, $\mathcal B \in \H[t]$ satisfies the Study condition if and only if $\mathcal B$ is vectorial (and will be denoted $\vec b$).
\end{lemma}

\begin{proof}
  The Study condition boils down to $0 = \alpha \mathcal A \Cj{(\mathcal B\mathcal A)} + \mathcal B\mathcal A \Cj{(\alpha \mathcal A)} = \alpha \mathcal A\Cj{\mathcal A} (\mathcal B + \Cj{\mathcal B})$. This implies $\mathcal B + \Cj{\mathcal B} = 0$, that is $\mathcal B$ is vectorial.
\end{proof}

\begin{definition}
  \label{def:framing}
  A motion polynomial $\mathcal C = (\alpha + \eps \vec b)\mathcal A$ (or the corresponding rational motion) is called \emph{framing} if the image $\mathcal A \qi \Cj{\mathcal A}$ of the vector $\qi$ is tangent to the trajectory $\vec{r}$ of the origin.
\end{definition}

Before we elucidate the relation of the framing motions to the PH curves let us give one more useful definition.

\begin{definition}
  We call the polynomial $\mathcal A \in \H[t]$ \emph{reduced with respect to $\qi \in \H$} if it is free of real factors of positive degree and of polynomial right factors with coefficients in the sub-algebra of $\H$ which is generated by $1$ and~$\qi$.
\end{definition}

\begin{theorem}
  \label{th:PH-framing}
  If $\mathcal C = (\alpha + \eps \vec b)\mathcal A$ is a framing motion then the trajectory $\vec{r}$ of the origin is a PH curve. Conversely, if $\vec{r}$ is a PH curve, there exists a framing motion $\mathcal C = (\alpha + \eps \vec b)\mathcal A$ of $\vec{r}$ where $\mathcal A$ is reduced with respect to~$\qi$.
\end{theorem}

\begin{proof}
  By Definition~\ref{def:framing}
  The motion $C$ is framing if and only if $\vec{r}'$ is linearly dependent to $\mathcal A \qi \Cj{\mathcal A}$. This implies
  \begin{equation*}
    \frac{\vec{r}'}{\Vert \vec{r}' \Vert} =
    \pm \frac{\mathcal A \qi \Cj{\mathcal A}}{\Vert \mathcal A \qi \Cj{\mathcal A} \Vert} =
    \pm \frac{\mathcal A \qi \Cj{\mathcal A}}{\mathcal A\Cj{\mathcal A}}.
  \end{equation*}
  From this we see that the speed function $\sigma$ of Definition~\ref{def:PH-curve} is rational and $\vec{r}$ is indeed a PH curve.

  Assume now conversely that $\vec{r}=\alpha^{-1} (r_1, r_2, r_3)$ is a rational
  PH curve and set $\vec b = -\frac{1}{2}(r_1\qi + r_2\qj + r_3\qk)$. The PH-property
  implies that the normalized hodograph $\vec{r}'/\Vert \vec{r}' \Vert$ is a
  rational spherical curve. It is well-known (c.f. for example the spherical
  specialization of \cite[Theorem~2]{li16}) that there exists $\mathcal A \in \H[t]$ such
  that $\vec{r}'/\Vert \vec{r}' \Vert = \mathcal A \qi \Cj{\mathcal A} / (\mathcal A\Cj{\mathcal A})$. Moreover
  there exist an $\mathcal A$ which is reduced with respect to $\qi$ and has this
  property. Indeed if $\mathcal R \in \H[t]$ has coefficients in the sub-algebra
  generated by $1$ and $\qi$ and is of maximal degree such that $\tilde{\mathcal A}=\mathcal A\mathcal R$
  then $\mathcal A \qi \Cj{\mathcal A}/(\mathcal A\Cj{\mathcal A}) =\tilde{\mathcal A} \qi \tilde{\Cj{\mathcal A}}/(\tilde{\mathcal A}
  \tilde{\Cj{\mathcal A}})$. Now $\mathcal C = (\alpha + \eps \vec b)\mathcal A$ is the sought framing motion.
\end{proof}

The previous theorem allows us to formulate three equivalent systems of linear equations characterizing rational PH curves.

\begin{theorem}
  \label{th:system-of-equations}
  Rational PH curves are precisely the trajectories of the origin under the
  rational motions $\mathcal C = (\alpha + \eps \vec b)\mathcal A$ where $\mathcal A$ is reduced with respect to $\qi$ and one of the following equivalent conditions is satisfied:
  \begin{enumerate}
  \item $(\alpha \vec b' - \alpha' \vec b) \times \mathcal A \qi \Cj{\mathcal A} = 0$
    \hfill\eqnum\label{eq:cross-product}
  \item $\langle \alpha \vec b' - \alpha'\vec b, \mathcal A \qj \Cj{\mathcal A} \rangle = \langle \alpha \vec b' - \alpha'\vec b, \mathcal A \qk \Cj{\mathcal A} \rangle = 0$
    \hfill\eqnum\label{eq:orthogonality}
  \item There exists $\mu \in \R[t]$ such that $\alpha \vec b' - \alpha' \vec b = \mu \mathcal A \qi \Cj{\mathcal A}$. \hfill\eqnum\label{eq:linear-dependence}
  \end{enumerate}
\end{theorem}

\begin{proof}
  Due to Theorem \ref{th:PH-framing} we only need to show that the three conditions are equivalent with $\mathcal C$ being framing. We have $\vec{r} = -2\alpha^{-1}\vec b$ implying $\vec{r}' = -2\alpha^{-2}(\alpha \vec b' - \alpha' \vec b)$. Therefore $\mathcal C$ is framing if and only if $\alpha \vec b' - \alpha' \vec b$ and $\mathcal A \qi \Cj{\mathcal A}$ are linearly dependent over the field of rational functions with real coefficients. Now, \eqref{eq:cross-product} expresses this linear dependency using the cross product. \eqref{eq:orthogonality} express the same property via orthogonality to the orthogonal complement $\Span(\mathcal A\qi \Cj{\mathcal A})^\perp= \Span(\mathcal A\qj \Cj{\mathcal A}, \mathcal A\qk \Cj{\mathcal A})$.

  In order to obtain \eqref{eq:linear-dependence} we express the linear dependence by existence of prime polynomials $\mu$, $\nu \in \R[t]$ such that
  \begin{equation*}
    \alpha \vec b' - \alpha' \vec b = \frac{\mu}{\nu} \mathcal A \qi \Cj{\mathcal A}.
  \end{equation*}
  Because the left-hand side of this equation is polynomial, $\nu$ must be a factor of $\mathcal A \qi \mathcal A^\star$. We will show that the assumption that $\mathcal A$ is reduced with respect to $\qi$ implies that $\mathcal A \qi \mathcal A^\star$ has no non-constant factors and therefore we have $\nu = 1$.

  Let us assume that $\psi$ is the unique monic real factor of maximal degree of $\mathcal A \qi \Cj{\mathcal A}$. By a simple but useful auxiliary result (Proposition~2.1 of \cite{cheng16} or Lemma~1 of \cite{li16}) and because $\mathcal A$ is assumed to have no real polynomial factor of positive degree, a positive degree of $\psi$ is only possible if there exists a linear quaternionic right factor $t-\mathcal{H}$ of $\mathcal A$ such that $t-\Cj{\mathcal{H}}$ is a left factor of $\qi \Cj{\mathcal A}$. In this case, $(t-\mathcal{H})(t-\Cj{\mathcal{H}})$ is a real quadratic factor of~$\psi$.

  The linear right factor $t-\mathcal{H}$ must be of a rather special shape. With $\mathcal{H} = h_0 + h_1\qi + h_2\qj + h_3\qk$ we find
  \begin{equation*}
    \qi(t - \Cj{\mathcal{H}}) = \qi(t - (h_0 - h_1\qi - h_2\qj - h_3\qk)) = (t - (h_0 - h_1\qi + h_2\qj + h_3\qk))\qi
  \end{equation*}
  because $t$, $h_0$, and $\qi$ commute with $\qi$ and $\qi\qj = -\qj\qi$,
  $\qi\qk = -\qk\qi$. This implies that $t - (h_0 - h_1\qi + h_2\qj + h_3\qk)$
  is a linear left factor of $\qi \Cj{\mathcal A}$. By the factorization theory of
  quaternionic polynomials \cite{hegedus13,li19}, linear left factors with a
  given norm polynomial are unique whence
  \begin{equation*}
    h_0 - h_1\qi - h_2\qj - h_3\qk = h_0 - h_1\qi + h_2\qj + h_3\qk
  \end{equation*}
  and consequently $h_2 = h_3 = 0$ -- a contradiction to the assumption that $\mathcal A$ is reduced with respect to~$\qi$.
\end{proof}
 
These equations and the relation between framing motions and PH curves stated in
Theorems~\ref{th:PH-framing}, \ref{th:system-of-equations} allow a computational
approach to rational PH curves. The basic idea is to prescribe $\alpha \in
\R[t]$ and $\mathcal A \in \H[t]$ and to compute the unknown coefficients of the
polynomial $\vec b \in \H[t]$ (subject to the constraint $\vec b + \Cj{\vec b} = 0$) from the
linear systems obtained by comparing coefficients of $t$ in
\eqref{eq:cross-product}, \eqref{eq:orthogonality} or
\eqref{eq:linear-dependence}. In most cases, the system obtained from
\eqref{eq:orthogonality} is to be preferred because it is the smallest, and the system obtained from \eqref{eq:linear-dependence} is to be avoided because it introduces auxiliary variables (the coefficients of $\mu$) that are not needed. However, for a more detailed analysis in Section~\ref{sec:existence} we will prefer system \eqref{eq:linear-dependence} which we found to provide the best general insight in solvability issues.

\section{Existence of Non-Polynomial Solutions}
\label{sec:existence}

Given $A \in \H[t]$, reduced with respect to $\qi$, and $\alpha \in \R[t]$ we
discuss existence of polynomial solutions $\vec b$ to one of the equivalent equation
systems of Theorem~\ref{th:system-of-equations}. Special emphasis is put on
solutions for $\vec b$ that lead to non-polynomial (truly rational) PH curves
$\vec{r} = -2\alpha^{-1}\vec b$. Our main results are summarized in
Theorem~\ref{th:non-polynomial-solutions} below.

It will make sense to consider also ``trivial'' solutions. They are characterized by constant $\vec{r}$ whence $\vec{r}' = 0$ which, of course, fulfills the PH condition. Trivial solutions have no direct relevance in applications. Nonetheless, the corresponding polynomials $\vec b \in \H[t]$ are elements of the solution space and can serve as elements of a basis, emphasizing the translation invariance of PH curves.

In case of the well-known polynomial PH curves, solving the systems of Theorem \ref{th:system-of-equations} can be circumvented by directly integrating $\vec{r}' = \lambda \mathcal A \qi \Cj{\mathcal A}$ with $\lambda \in \R[t]$, c.f. \eqref{racint}. For $\lambda = 0$ we recover trivial solutions (which are polynomial as well), for $\lambda \neq 0$ we will speak of ``non-trivial polynomial solutions''.

\begin{lemma}
  \label{lem:polynomial}
  A solution to the systems of Theorem~\ref{th:system-of-equations} gives rise
  to a polynomial PH curve if and only if $\alpha \in \R[t]$ is a factor of~$\vec b$.
  Trivial solutions are obtained precisely for $\vec b = \alpha \vec b_0$ where $\vec b_0 \in
  \H$ is constant and vectorial.
\end{lemma}

\begin{proof}
  From $\vec b = \alpha \tilde{\vec b}$ with $\tilde{\vec b} \in \H[t]$ we obtain $\vec{r} = -2\alpha^{-1}\vec b = -2\tilde{\vec b}$ and the curve is indeed polynomial. Conversely, if $\vec{r} = -2\alpha^{-1}\vec b$ is polynomial, then clearly $\alpha$ is a factor of $\vec b$. The statement on the trivial solutions is trivial.
\end{proof}

Since polynomial PH curves are well-understood, we are mostly interested in
truly rational (non-polynomial) solutions. Moreover, it is not our immediate aim
to compute solutions of \eqref{eq:cross-product}, \eqref{eq:orthogonality}, or
\eqref{eq:linear-dependence} but to analyze existence and type of solutions.
Questions about dimension and basis of the solution space will not be formally
treated in this paper but some observations can be found in
Section~\ref{sec:examples}.

The system \eqref{eq:linear-dependence} is well-suited for studying solvability
and in particular the type of solutions (trivial, polynomial, or
non-polynomial). Rather surprisingly, it turns out, that the desired rational
solutions only occur in exceptional cases. We proceed by analysing in detail the
system of linear equations \eqref{eq:linear-dependence}. In particular we will
fully decide for which inputs $A$ and $\alpha$ there exist solutions $\vec b$ for
which the resulting PH curve is truly rational (non-polynomial).

Our main technical tool consists in fixing a monic linear polynomial $\beta \in \C[t]$ and considering \eqref{eq:linear-dependence} in the polynomial basis $(1,\beta,\beta^2,\beta^3,\ldots)$. As we will see, in the interesting cases $\beta$ will be a factor of $\alpha$. For this reason we consider the complex extension in order to handle the non-real roots of $\alpha$ as well. If $\beta \in \R[t]$ is real, all the coefficients introduced below are real as well. Otherwise, the coefficients are complex but a simple additional argument in the proof of Theorem~\ref{th:non-polynomial-solutions} will guarantee existence of real solutions.

Suppose, that $\beta$ is a factor of $\alpha$ of multiplicity $n\ge 0$. We
express the polynomials $\vec b$, $\mathcal A \qi \mathcal A^\ast$, $\mu$ and $\gamma \coloneqq \alpha
/ \beta^n$ in the basis $(1,\beta,\beta^2,\beta^3,\ldots)$:
\begin{equation*}
  \vec b = \sum_i \beta^i \vec b_i,\quad
  \mathcal A \qi \Cj{\mathcal A} = \sum_i \beta^i \vec f_i,\quad
  \mu = \sum_i \beta^i \mu_i,\quad
  \gamma = \sum_i \beta^i \gamma_i.
\end{equation*}
The vectorial coefficients $\vec f_i$ and $\gamma_i$ are determined by the input
data. Note that $\vec f_0 \neq 0$ because $\mathcal A$ is reduced with respect to $\qi$ and
$\gamma_0 \neq 0$ because $n$ is the multiplicity of $\beta$ as factor of
$\alpha$. The unknowns are $\vec \vec b_i$ which, by Lemma~\ref{lem:vectorial}, are
vectorial quaternions and the scalars $\mu_i$. We are looking for solution
polynomials, whence only finitely many of the coefficients $\vec \vec b_i$, $\mu_i$ can be
different from zero. However, at this point, we do not wish to bound their
respective degrees. Thus, we do not restrict the range of the summation variable
$i \in \mathbb{N}_0$. With regard to \eqref{eq:linear-dependence}, we compute
the expressions $\alpha \vec b' - \alpha' \vec b$ and $\mu \mathcal A \qi \mathcal A^\ast$:
\begin{equation*}
  \begin{aligned}
    \alpha \vec b' - \alpha' \vec b &= \beta^n \gamma \vec b' - \beta^n \gamma' \vec b - n \beta^{n-1} \gamma \vec b \\
    &= \beta^n (\sum_i \beta^i \gamma_i) (\sum_i i \beta^{i-1} \vec b_i) - \beta^n (\sum_i i \beta^{i-1} \gamma_i) (\sum_i \beta^i \vec b_i) - n\beta^{n-1} (\sum_i \beta^i \gamma_i) (\sum_i \beta^i \vec b_i) \\
    &= \sum_{i,j} \beta^{i+j+n-1} (j-i-n) \gamma_i \vec b_j \\
    \mu \mathcal A \qi \mathcal A^\ast &= (\sum_i \beta^i \mu_i) (\sum_i \beta^i \vec f_i) = \sum_{i,j} \beta^{i+j} \mu_i \vec f_j.
  \end{aligned}
\end{equation*}
Comparing coefficients in this basis, system \eqref{eq:linear-dependence} becomes
\begin{equation}
  \label{eq:linear-dependence-system}
  \sum_{i+j=k-n} (j-i-n) \gamma_i \vec b_j = \sum_{i=0}^{k-1} \mu_i \vec f_{k-i-1},\quad
  k \ge 1.
\end{equation}

This is a system of homogeneous linear equations with unknowns $\vec b_i$ and
$\mu_i$. In the following lemmas, we will exploit the particular (triangular and
symmetric) form of \eqref{eq:linear-dependence-system} to analyse the solution
space depending on $n$. Let us observe that in all cases $\vec f_0 \neq 0$ by the
assumption that $\mathcal A$ is reduced with respect to $\qi$ and $\gamma_0 \neq 0$ as
otherwise the multiplicity of $\beta$ as factor of $\alpha$ would be larger than
$n$. Our main purpose is to discuss existence or non-existence of solutions for
$\vec b$ having $\beta$ as a factor with certain multiplicity. These can be very
simply identified via the annulation of initial coefficients: By
Lemma~\ref{lem:polynomial}, solutions with $(\vec b_0,\vec b_1,\ldots,\vec b_{n-1}) \neq (0, 0,
\ldots, 0)$ give rise to non-polynomial PH curves.

\begin{lemma}
  \label{lem:factors-1}
  For given $\alpha \in \R[t]$ and $\mathcal A \in \H[t]$, reduced with respect to $\qi$,
  consider a monic linear factor $\beta \in \C[t]$ of $\alpha$ of multiplicity
  $n = 1$. There exists a polynomial solution $\mu$, $\vec b$ of
  \eqref{eq:linear-dependence} such that $\beta$ does not divide $\vec b$ if and only
  if $\{\vec f_0,\vec f_1\}$ are linearly dependent.
\end{lemma}

\begin{proof}
  For $n = 1$ the system \eqref{eq:linear-dependence-system} boils down to
  \begin{equation}
    \label{eq:case1}
    \begin{aligned}
                                                - 1\gamma_0\vec b_0 &= \mu_0\vec f_0,\\
                                   0\gamma_0\vec b_1 - 2\gamma_1\vec b_0 &= \mu_0\vec f_1 + \mu_1\vec f_0,\\
                    1\gamma_0\vec b_2 - 1\gamma_1\vec b_1 - 3\gamma_2\vec b_0 &= \mu_0\vec f_2 + \mu_1\vec f_1 + \mu_2\vec f_0,\\
     2\gamma_0\vec b_3 + 0\gamma_1\vec b_2 - 2\gamma_2\vec b_1 - 4\gamma_3\vec b_0 &= \mu_0\vec f_3 + \mu_1\vec f_2 + \mu_2\vec f_1 + \mu_3\vec f_0,\\
     &\;\;\vdots
    \end{aligned}
  \end{equation}
  Contrary to common convention, we do not omit the coefficients $0$ and $1$ on the left-hand side in order to highlight the system's regular structure.

  If $\{\vec f_0,\vec f_1\}$ is linearly dependent, we construct a solution with $\vec b_0\neq 0$. The first vectorial equation is satisfied setting $\vec b_0 = -\mu_0\gamma_0^{-1}\vec f_0$. Plugging this into the second equation yields
  \begin{equation}
    \label{eq:mu-1}
    \mu_0\vec f_1 + (\mu_1 - 2\gamma_1\mu_0\gamma_0^{-1})\vec f_0 = 0.
  \end{equation}
  Since $\vec f_0$ and $\vec f_1$ are linearly dependent and $\vec f_0\neq 0$, there exists $\varphi \in \R$ so that $\vec f_1+\varphi \vec f_0 = 0$. We set $\mu_0 = 1$ and $\mu_1 = \varphi + 2\gamma_1\gamma_0^{-1}\varphi_1$ so that the first and second equation are satisfied.

  We still need to argue that there exists a solution for the remaining
  variables $\vec b_1$, $\vec b_2$, \ldots{} and $\mu_2$, $\mu_3$, \ldots{} with only
  finitely many of the vectorial quaternions $\vec b_i$ different from zero.

  The further equations can be satisfied, because $\vec b_2, \vec b_3, \ldots$ appear for the first time one by one and with the nonzero coefficient $\gamma_0$. We may therefore arbitrarily assign $\vec b_1$, $\mu_2$, $\mu_3$, \ldots, and solve the remaining vectorial equations successively for $\vec b_2$, $\vec b_3$, \ldots In particular, we are free to select only finitely many $\mu_i$ different from zero. Then there exists $M \in \mathbb{N}$ such that $\mu_m = \gamma_m = \vec f_m = 0$ for all $m \ge M$. This implies existence of a solution where only finitely many of the vectorial quaternions $\vec b_i$ are different from zero as well. By construction, in this solution $\vec b_0 = -\gamma_0^{-1}\vec f_0\neq 0$ and thus, $\beta$ does not divide~$\vec b$.

  If the set $\{\vec f_0,\vec f_1\}$ is linearly independent, \eqref{eq:mu-1} implies
  $\mu_0 = \mu_1 = 0$ whence $\vec b_0 = 0$ and $\beta$ divides all possible
  solutions for~$\vec b$. The argument for existence of polynomial solutions is the
  same as above.
\end{proof}

\begin{lemma}
  \label{lem:factors-2}
  For given $\alpha \in \R[t]$ and $\mathcal A \in \H[t]$, reduced with respect to $\qi$,
  consider a monic linear factor $\beta \in \C[t]$ of $\alpha$ of multiplicity
  $n = 2$. There exists a polynomial solution $\mu$, $\vec b$ of
  \eqref{eq:linear-dependence} such that $\beta^2$ does not divide $\vec b$ if and
  only if $\{\vec f_0,\vec f_1,\vec f_2\}$ are linearly dependent.
\end{lemma}

\begin{proof}
  For $n = 2$ the system \eqref{eq:linear-dependence-system} reads as
  \begin{equation}
    \label{eq:case2}
    \begin{aligned}
      0 &= \mu_0 \vec f_0, \\
      -\;2 \gamma_0 \vec b_0 &= \mu_0 \vec f_1 + \mu_1 \vec f_0, \\
      -\;1\gamma_0 \vec b_1 - 3 \gamma_1 \vec b_0 &= \mu_0 \vec f_2 + \mu_1 \vec f_1 + \mu_2 \vec f_0, \\
      0\gamma_0\vec b_2 - 2 \gamma_1 \vec b_1 - 4 \gamma_2 \vec b_0 &= \mu_0 \vec f_3 + \mu_1 \vec f_2 + \mu_2 \vec f_1 + \mu_3 \vec f_0, \\
      1\gamma_0 \vec b_3 - 1\gamma_1\vec b_2 - 3\gamma_2\vec b_1 - 5\gamma_3\vec b_0 &= \mu_0 \vec f_4 + \mu_1 \vec f_3 + \mu_2 \vec f_2 + \mu_3 \vec f_1 + \mu_4 \vec f_0, \\
      &\;\;\vdots
    \end{aligned}
  \end{equation}
  Since $\vec f_0 \neq 0$, the first equation implies $\mu_0 = 0$ canceling the first terms of all right-hand sides. Solving the second and the third equation for $\vec b_0$ and $\vec b_1$ (recall that $\gamma_0 \neq 0$) and plugging the result into the fourth equation yields a relation
  \begin{equation}
    \label{eq:mu-2}
    \mu_1 \vec f_2 + (\mu_2 + \varrho_{11}\mu_1) \vec f_1 +(\mu_3 + \varrho_{20}\mu_2 + \varrho_{10}\mu_1)\vec f_0 = 0
  \end{equation}
  with some coefficients $\varrho_{ij}$ which could be expressed explicitly but there is no need for their particular form.

  If $\{\vec f_0,\vec f_1,\vec f_2\}$ is linearly independent we infer $\mu_1 = \mu_2 = \mu_3 = 0$ whence $\vec b_0 = \vec b_1 = 0$ and $\beta^2$ necessarily divides all possible solutions $\vec b$.

  If, on the other hand, $\{\vec f_0,\vec f_1,\vec f_2\}$ is linearly dependent, there exists a
  triple $(\varphi_0,\varphi_1,\varphi_2) \neq (0,0,0)$ of complex numbers such
  that $\varphi_0\vec f_0 + \varphi_1\vec f_1 + \varphi_2\vec f_2 = 0$. With $\mu_1 =
  \varphi_2$, $\mu_2 = \varphi_1 - \varrho_{11}\mu_1$, and $\mu_3 = \varphi_0 -
  \varrho_{20}\mu_2 - \varrho_{10}\mu_1$, the fourth equation is fulfilled and
  $(\mu_1,\mu_2,\mu_3) \neq (0,0,0)$. The corresponding solutions for $\vec b_0$,
  $\vec b_1$ can not be simultaneously zero because in such case the left-hand side
  of the first four equations would be zero, forcing $\mu_0=\mu_1=\mu_2=\mu_3=0$.
  Solutions for $\vec b_2$, $\vec b_3$, \ldots\ are constructed in the same way as in the
  proof of Lemma \ref{lem:factors-1}. More precisely, we may arbitrarily assign $\vec b_2$, $\mu_4$, $\mu_5$, \ldots, and solve the remaining vectorial equations successively for $\vec b_3$, $\vec b_4$, \ldots\ If only finitely many of $\mu_4$, $\mu_5$, \ldots\ are different from zero, a solution with only finitely many $\vec b_2$, $\vec b_3$, \ldots\ different from zero exists. In this solution, $\beta^2$ does not divide~$\vec b$.
\end{proof}

\begin{lemma}
  \label{lem:factors-n}
  For given $\alpha \in \R[t]$ and $\mathcal A \in \H[t]$, reduced with respect to $\qi$,
  consider a monic linear factor $\beta \in \C[t]$ of $\alpha$ of multiplicity
  $n \ge 3$. Then there always exists a polynomial solution $\mu$, $\vec b$ of \eqref{eq:linear-dependence} such that $\beta^n$ does not divide $\vec b$.
\end{lemma}

\begin{proof}
  Let us describe the system \eqref{eq:linear-dependence-system} for a general $n\ge 3$. The right-hand sides of all equations have a very regular triangular structure. In the $i$-th equation one new variable $\mu_{i+1}$ is introduced with the non-zero coefficient~$\vec f_0$.

  The left hand sides are slightly more complicated. The left-hand side of the first $n-1$ equations equals $0$. In the $n$ subsequent equations $n$ variables $\vec b_0$, $\vec b_1$, \ldots, $\vec b_{n-1}$ are introduced one by one with non-zero coefficients which is an integer multiple of $\gamma_0$. The next equation does not introduce any new variable on the left-hand side. All the remaining equations have a very regular left-hand side introducing always one new variable $\vec b_{n+i}$ with non-zero coefficient in the $(2n+i)$-th equation.

  Based on this general structure we can always construct a non-polynomial
  solution curve in the following way.

  \begin{enumerate}
  \item Satisfy the equations number $1, \ldots, (n-1)$ by setting $\mu_0 = \mu_1 = \cdots = \mu_{n-2} = 0$.
  \item Satisfy the equations number $n, \ldots, (2n-1)$ by expressing the unknowns $\vec b_0$, $\vec b_1$, \ldots, $\vec b_{n-1}$ uniquely in terms of the input coefficients ($\vec f_i$, $\gamma_i$) and the unknowns $\mu_{n-1}$, $\mu_n$, \ldots, $\mu_{2n-2}$ (which remain free).
  \item \label{step:3} Substitute the previous expressions to the equation number $2n$ to obtain one homogeneous vector equation in the $(n+1)$ scalar unknowns $\mu_{n-1}$, $\mu_n$, \ldots, $\mu_{2n-1}$. Set these unknowns to a non-trivial solution which must exist because of $n \ge 3$.
  \item Set $\vec b_n$ and $\mu_{2n}$, $\mu_{2n+1}$, \ldots arbitrarily and solve the remaining equations uniquely for $\vec b_{n+1}$, $\vec b_{n+2}$, \ldots If only finitely many of $\mu_{2n}$, $\mu_{2n+1}$, \ldots are different from zero, a solution with only finitely many of $\vec b_{n+1}$, $\vec b_{n+2}$, \ldots different from zero does exist.
  \end{enumerate}
  We claim that in this solution $\vec b$ does not have $\beta^n$ as a factor. Indeed, if all $\vec b_0$, $\vec b_1$, \ldots, $\vec b_{n-1}$ are simultaneously zero then $\mu_{n-1} = \mu_n = \cdots = \mu_{2n-1} = 0$ -- via equations $n, \ldots, 2n$ -- contradicting the non-triviality of the solution in Step~\ref{step:3}.
\end{proof}

Before combining Lemmas~\ref{lem:factors-1}--\ref{lem:factors-n} into the central result of this article, we provide a simple observation that allows us to infer existence of real solutions even if these lemmas allow for $\beta \in \C[t]$.

\begin{lemma}
  \label{lem:rational}
  If $b$ is a rational function over $\C$ but not polynomial, then either its
  real part $\frac{1}{2}(b+\overline{b})$ or its imaginary part
  $\frac{1}{2}(b-\overline{b})$ is non-polynomial as well. (The bar denotes
  complex conjugation.)
\end{lemma}

\begin{proof}
  If both, real and imaginary part, were polynomial, then so is~$b$.
\end{proof}

\begin{theorem}
  \label{th:non-polynomial-solutions}
  For given $\alpha \in \R[t]$ and $\mathcal A \in \H[t]$, reduced with respect to $\qi$ the following two conditions are equivalent:
  \begin{enumerate}
  \item There exist a polynomial $ \vec b\in \H[t]$ so that $\mathcal C = (\alpha + \eps \vec b)\mathcal A$ is a motion polynomial and the trajectory of the origin is a rational non-polynomial PH curve.
  \item There exists a (real or complex) root $z$ of $\alpha$ of multiplicity $n$ such that the set $\{\vec f_0,\vec f_1,\ldots,\vec f_n\}$ with coefficients $\vec f_i$ defined by $\mathcal A \qi \mathcal A^\ast = \sum_i (t-z)^i \vec f_i$ is linearly dependent. Note that the linear dependence condition is always satisfied if $n \ge 3$.
  \end{enumerate}
\end{theorem}

\begin{proof}
  For a real zero $z$ the theorem's statement follows from Theorem
  \ref{th:system-of-equations} and
  Lemmas~\ref{lem:factors-1}--\ref{lem:factors-n}. In case of $z \in \C
  \setminus \R$, the complex conjugate $\overline{z}$ is also a zero of
  multiplicity $n$ of $\alpha$ and the solutions $\vec b$ and $\overline{\vec b}$ of
  \eqref{eq:linear-dependence} to $z$ and $\overline{z}$, respectively, are
  interchanged by complex conjugation. Their linear combinations $\vec b +
  \overline{\vec b}$ and $\vec b - \overline{\vec b}$ provide solutions for
  \eqref{eq:cross-product} or \eqref{eq:orthogonality} as well. They are real
  and, by Lemma~\ref{lem:rational}, at least one of them leads to a
  non-polynomial PH curve.
\end{proof}

\begin{remark}
  \label{rem:polynomial-solutions}
  Theorem~\ref{th:non-polynomial-solutions} talks only about existence of
  non-polynomial solution curves but does not exclude polynomial solutions. In fact, trivial solutions always exist and non-trivial polynomial solutions exist if the degree of $\vec b$ is high enough. Lemma~\ref{lem:polynomial} ensures polynomiality if $\alpha$ is a factor of $\vec b$. This condition can be encoded by linear homogeneous constraints on the coefficients of $\vec b$ that can be added, for example, to the system \eqref{eq:orthogonality}. This is an alternative to the usual methods of computing polynomial PH curves.
\end{remark}

Theorem~\ref{th:non-polynomial-solutions} provides necessary and sufficient
conditions for the existence of non-polynomial solutions to data $\alpha$,
$\mathcal A$ but does not guarantee that a particular set of conditions can
actually be met by suitable $\alpha$ and $\mathcal A$. This is not logically
needed for the correctness of Theorem~\ref{th:non-polynomial-solutions} but,
nonetheless, should be clarified. For the construction of non-polynomial
solutions it is also of interest to understand the geometry of the exceptional
cases of Lemma~\ref{lem:factors-1} and Lemma~\ref{lem:factors-2}.

The vector valued polynomial function $\vec F \coloneqq \mathcal A \qi \Cj{\mathcal A}$ can be
interpreted as parametric equation of a curve in the projective plane
$\mathbb{P}^2(\R)$. We may consider it as the projection of the PH curve's
tangent indicatrix into the plane at infinity from the coordinate origin. With
this understanding, conditions on linear dependence in
Lemmas~\ref{lem:factors-1} and \ref{lem:factors-2}, respectively, allow a
geometric interpretation. Denote by $z$ the unique (real or complex) zero of the
linear factor $\beta$ of $\alpha$. Then the vectors $\vec F(z)$, $\vec F'(z)$, and $\vec F''(z)$ are non-zero multiples of $\vec f_0$, $\vec f_1$, and $\vec f_2$, respectively. In other words, linear dependence of $\{\vec f_0,\vec f_1\}$ is equivalent to linear dependence of $\{\vec F(z), \vec F'(z)\}$ and linear dependence of $\{\vec f_0,\vec f_1,\vec f_2\}$ is equivalent to linear dependence of $\{\vec F(z),\vec F'(z),\vec F''(z)\}$.

For real $z$, above linear dependencies generically have the following geometric
meaning:
\begin{itemize}
\item The set $\{\vec F(z),\vec F'(z)\}$ is linearly dependent if the point $\vec F(z)$ is a
  cusp of the rational parametric curve $\vec F$, c.f. \cite[p.~2]{bol}.
\item The set $\{\vec F(z),\vec F'(z),\vec F''(z)\}$ is linearly dependent if the point $\vec F(z)$
  is an inflection point of the rational parametric curve $\vec F$, c.f. \cite[p.~4]{bol}.
\end{itemize}

The construction of examples with inflection points is straightforward. Starting
with a polynomial $\mathcal A \in \H[t]$, we define as usual $\vec F \coloneqq \mathcal A \qi \Cj{\mathcal A}$
and compute the polynomial $D \coloneqq \det(\vec F,\vec F',\vec F'')$. The necessary and
sufficient condition for linear dependence of $\{\vec f_0,\vec f_1,\vec f_2\}$ is that $\beta$
is a linear factor of $D$. It can easily be fulfilled (c.f.
Example~\ref{ex:n2}).

In order to demonstrate feasibility of the cusp condition, we write $\mathcal A = \sum_i \mathcal A_i t^i$ with yet undetermined quaternion coefficients $\mathcal A_i$ and compute $\vec F \coloneqq \mathcal A \qi \Cj{\mathcal A} = \sum_i \vec f_i t^i$, where $\vec f_0 = \mathcal A_0 \qi \Cj{\mathcal A_0}$ and $\vec f_1 = \mathcal A_1\qi\Cj{\mathcal A_0}+\mathcal A_0\qi\Cj{\mathcal A_1}$. Solutions to the system of algebraic equations arising from
\begin{equation}
  \label{eq:lindep}
  (\mathcal A_0\qi\Cj{\mathcal A_0}) \times (\mathcal A_1\qi\Cj{\mathcal A_0} + \mathcal A_0\qi\Cj{\mathcal A_1}) = 0
\end{equation}
will result in a cusp at $t = 0$ (which is as good as any other real parameter
value for our purpose). It is no loss of generality to assume $\mathcal A_0 = 1$ whence
the cusp condition \eqref{eq:lindep} implies that $\mathcal A_1$ can be expressed as $\mathcal A_1
= a_{10} + a_{11}\qi$ with $a_{10}$, $a_{11} \in \R$. Determining examples with
cusps at complex parameter values is more involved. One particular case, taken
from \cite{krajnc1}, is provided in Example~\ref{ex:krajnc1} below.

\section{Examples and Observations}
\label{sec:examples}

We continue this article by presenting examples, some of them taken from
existing literature. The discussion will not only show how the criteria of
Theorem~\ref{th:non-polynomial-solutions} are met in known examples but also suggest interesting properties of the solution space.

\begin{example}
  \label{ex:simple}
We consider the polynomials $\mathcal A = t^2-(\qi+\qj)t+\qk$ and $\alpha = t^3$ whence necessarily $\beta = t$ and $n = 3$. We are in the case of Lemma~\ref{lem:factors-n} so that non-polynomial solutions do exist. Solving the the system \eqref{eq:orthogonality} for different degrees of $\vec b$ we find the following: In case of $\deg \vec b < 2$, the only solution is $\vec r_1(t) = (0, 0, 0)^T$. In case of $2 \le \deg \vec b \le 5$ we obtain the trivial solution $\vec r_2(t)=(\tau_1,\tau_2,\tau_3)^T = \vec r_1(t) + (\tau_1,\tau_2,\tau_3)^T$ with constant $\tau_1$, $\tau_2$, $\tau_3 \in \R$. For $6 \le \deg \vec b \le 7$ the solutions space is of dimension four and consists of scalings and translations $\vec r_6(t) = \mu_6 \vec p_6(t) + (\tau_1,\tau_2,\tau_3)^T$ of one basis curve
  \begin{equation*}
    \vec p_6(t) = \frac{1}{t^3}
    \begin{pmatrix}
      t^6+3t^4+t^3+3t^2+1\\
      12t^4+t^3-12t^2\\
      3t^5+t^3+3t
    \end{pmatrix}.
  \end{equation*}
  Note that non-trivial polynomial PH curves are not among the solutions found thus far as they belong to $\deg \vec b \ge 8$. In case of $\deg \vec b = 8$, the solution space is of dimension five, the corresponding rational PH curves are given by $\vec r_8(t) = \nu_8 \vec p_8(t) + \vec r_6(t)$ where $\vec p_8(t) = \int \mathcal A \qi \Cj{\mathcal A} \;\mathrm{d}t$ is the usual quintic polynomial PH curve. The solution space for $\deg \vec b = 9$ is of dimension six, the corresponding rational PH curves are given by $\vec r_9(t) = \nu_9 \vec p_9(t) + \vec r_8(t)$ where $\vec p_9(t) = \int t \mathcal A \qi \Cj{\mathcal A} \;\mathrm{d}t$.
\end{example}

\begin{remark}
  Not unexpectedly, the recursion of Example~\ref{ex:simple} continues as $\vec r_m(t) = \nu_m\vec p_m(t) + \vec r_{m-1}(t)$ where $\vec p_m(t) = \int t^{m-8} \mathcal A \qi \Cj{\mathcal A} \;\mathrm{d}t$. We refrain from formally proving this in this paper. A similar statement for $\alpha$ having several linear factors is a topic of future research. These observations hint at the existence of nested solution spaces with a clear recursive generation.
\end{remark}

\begin{example}
  \label{ex:n2}
  In this example we pick again $\mathcal A  = t^2 - (\qi + \qj)t + \qk$ but select
  $\alpha = (t+1)^2$. With $\vec F \coloneqq \mathcal A \qi
  \mathcal A^\ast$ and $\beta = t+1$ we have $n=2$ and $\vec F = \vec f_0 + \vec f_1\beta +
  \vec f_2\beta^2 + \vec f_3\beta^3 + \vec f_4\beta^4$ where
  \begin{gather*}
    \vec f_0 = 4 \qj,\quad
    \vec f_1 = -4 \qi - 8 \qj + 4 \qk,\quad
    \vec f_2 = 6 \qi + 4 \qj - 6 \qk,\quad
    \vec f_3 = -4 \qi + 2 \qk,\quad
    \vec f_4 = \qi.
  \end{gather*}
  We observe that $\det(\vec f_0,\vec f_1,\vec f_2) = 0$ so that we can indeed predict rational solutions by Lemma~\ref{lem:factors-2}. Computing solutions for increasing degree of $\vec b$ we find $\vec r_0(t) = \vec r_1(t) = (0,0,0)^T$ for $0 \le \deg \vec b \le 1$. For $2 \le \deg \vec b \le 5$ we find the trivial solutions $\vec r_2(t) = \vec r_1(t) + (\tau_1,\tau_2,\tau_3) = (\tau_1,\tau_2,\tau_3)$. The first non-trivial solutions occur for $\deg \vec b = 6$. They are given by $\vec r_3(t) = \mu_3 \vec r(t) + \vec r_2(t)$, that is, they are obtained from the prototype solution
  \begin{equation*}
    \vec r(t) = \frac{1}{(t+1)^2}
    \begin{pmatrix}
      -6t^6-10t^5-2t^4+8t^3-2t^2-70t-62\\
      -48t^4-72t^3\\
      -16t^5-26t^4+44t^3-108t-54
  \end{pmatrix}
  \end{equation*}
  by scaling and translation. We were not aware of an instance of a rational PH
  curve satisfying the condition of Lemma~\ref{lem:factors-2} in prior
  literature. The rational PH curve $\vec r(t)$ and a quintic polynomial PH
  curve $\vec p(t)$ with identical start point $\vec{r}(0) = \vec{p}(0)$ are
  displayed in Figure~\ref{fig:1}. The respective end points $\vec r(1)$ and
  $\vec p(1)$ are different and suitable linear combinations of $\vec r(t)$ and
  $\vec p(t)$ would allow us to reach further end points. Adding more
  rational PH curves with the same tangent indicatrix increases the vector space
  of solutions and would allow to solve various interpolation problems. The
  right-hand side in Figure~\ref{fig:1} shows the portion of the tangent
  indicatrix over the parameter interval $[-3,0.5]$. As expected, the point to
  parameter value $t = -1$ is, indeed, a spherical inflection point. The
  inflection tangent in this geometry, a great circular arc, is displayed as
  well.
\end{example}

\begin{figure}
  \centering
  \begin{minipage}{0.55\linewidth}
    \centering
    \includegraphics[]{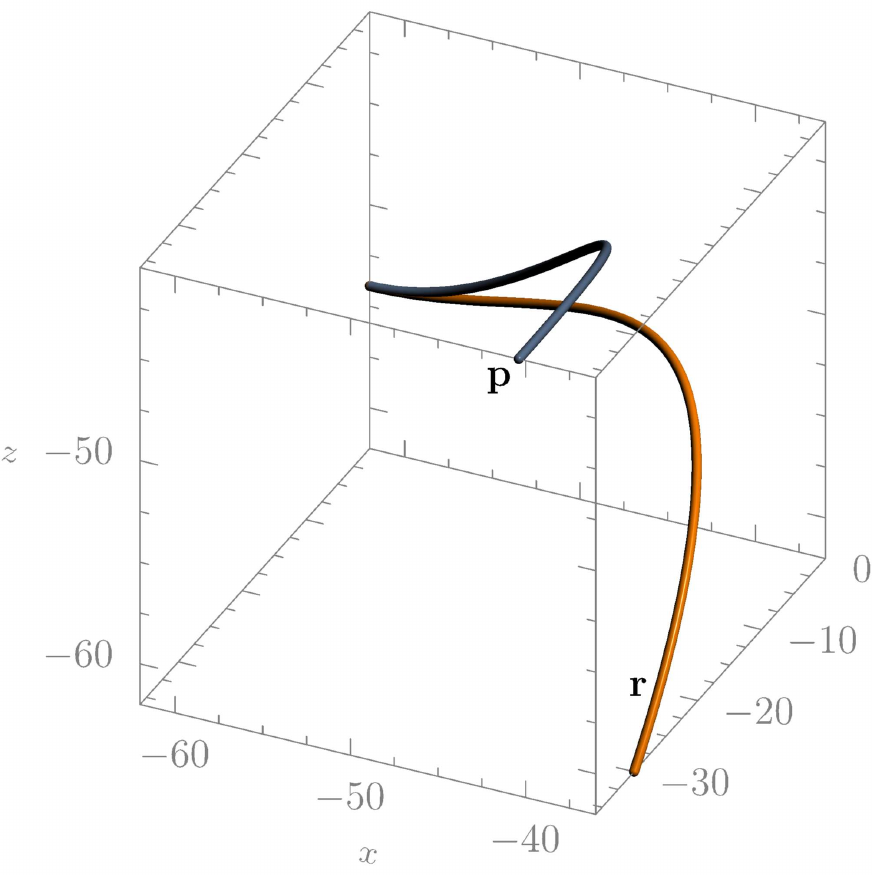}
  \end{minipage}\begin{minipage}{0.45\linewidth}
    \centering
    \includegraphics[]{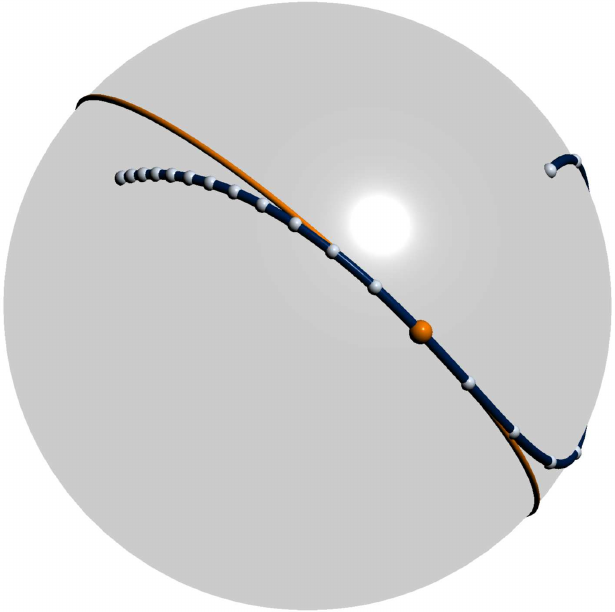}
  \end{minipage}\caption{Rational PH curve $\vec r$, polynomial PH curve $\vec p$ (left) and
    their common spherical tangent indicatrix (right). While the PH curves are
    parametrized over $[0,1]$, the parameter interval for the tangent indicatrix is
    $[-3,0.5]$ in order to illustrate the inflection point at $t = -1$.}
  \label{fig:1}
\end{figure}

\begin{example}
  \label{ex:krajnc1}
  In \cite[Theorem~7]{krajnc1} the authors present a family of rational PH curves of the particularly low degree three. One example of theirs is
  \begin{equation}
    \label{eq:ex-krajnc1}
    \vec r(t) = \frac{-1}{60(t^2+1)}
    \begin{pmatrix}
      t(t^2-4)\\
      2t(3t-1)\\
      t(3t+4)
    \end{pmatrix}.
  \end{equation}
  In order to reconstruct it, we use $\alpha = 60(t^2+1)$ and $\mathcal A = t^2+3t\qi
  +2\qj+\qk-1$ which is provided by \cite[Theorem~7]{krajnc1} as well. With $\vec F
  \coloneqq \mathcal A \qi \mathcal A^\ast$ and $\beta = t - \ci$ we have $n = 1$ and $\vec F = \vec f_0 +
  \vec f_1\beta + \vec f_2\beta^2 + \vec f_3\beta^3 + \vec f_4\beta^4$ where
  \begin{gather*}
    \vec f_0 = -10\qi - (4 - 12\ci)\qj + (8 + 6\ci)\qk,\quad
    \vec f_1 = (10\ci)\qi + (12 + 4\ci)\qj + (6 - 8\ci)\qk,\\
    \vec f_2 = \qi + 2\qj - 4\qk,\quad
    \vec f_3 = (4\ci)\qi,\quad
    \vec f_4 = \qi.
  \end{gather*}
  We observe that $\rank(\vec f_0,\vec f_1) = 1$ so that we can indeed predict non-polynomial solutions by Lemma~\ref{lem:factors-1}. Let us proceed by computing solutions for increasing degree of $\vec b$. For $0 \le \deg \vec b \le 1$ we obtain $\vec r_0(t) = \vec r_1(t) = (0,0,0)^T$. Already for $\deg \vec b = 2$ we find the translations $\vec r_2(t) = \vec r_1(t) + (\tau_1,\tau_2,\tau_3)^T = (\tau_1,\tau_2,\tau_3)$ of $\vec r_1(t)$. The first non-trivial solution occurs indeed for $\deg \vec b = 3$. The solution family is obtained from \eqref{eq:ex-krajnc1} by scaling and translation: $\vec r_3(t) = \mu_3\vec r(t) + \vec r_2(t)$. Note that our proof of existence of non-polynomial solutions requires complex numbers but actually real solutions are computed from a real system of linear equations.
\end{example}

\begin{example}
  In \cite[Example~1]{FaroukiSir2} a rational PH curve
  \begin{equation}
    \label{eq:ex-FaroukiSir2}
    \vec r(t)=\frac{800}{(t+10)^5}
    \begin{pmatrix}
      -13420\,t^5+4000\,t^4+30000\,t^3-50000\,t^2+25000\,t\\
      46643\,t^5-67850\,t^4-7000\,t^3+30000\,t^2\\
      19776\,t^5-111200\,t^4+126000\,t^3-40000\,t^2
    \end{pmatrix}
  \end{equation}
  is constructed using essentially formula \eqref{eq:tr2}. We have $\alpha = (t+10)^5$, $\beta = t + 10$ and $n = 5$ whence we are in the case of Lemma~\ref{lem:factors-n} again. The spherical motion of the Euler-Rodrigues frame is provided by \cite{FaroukiSir2}:
  \begin{equation*}
    \mathcal A=(7 t^2-22 t+10)+(-19 t^2+14 t)\qi+(-26 t^2+16 t)\qj+(-2t^2+12 t)\qk.
  \end{equation*}
  We have $\vec F \coloneqq \mathcal A \qi \mathcal A^\ast = \sum_{i=0}^4 \vec f_i\beta^i$. Using the system \eqref{eq:orthogonality} we can now compute solutions for $\vec b$. Already for $\deg \vec b = 4$ we find the non-polynomial solution
  \begin{equation}
    \label{eq:ex1-r4}
    \vec p_4(t) =
    \frac{1}{(t+10)^5}
    \begin{pmatrix}
      -27t^4-538t^3-5366t^2-26841t-53680\\
      96t^4+1866t^3+18656t^2+93286t+186572\\
      44t^4+786t^3+7912t^2+39552t+79104
    \end{pmatrix}.
  \end{equation}
  Obviously, every uniform scaling of $\vec p_4(t)$ is a solution as well. For $\deg
  \vec b = 5$ we obtain a solution space of dimension four, consisting of the
  scalings and translations of $\vec p_4(t)$. The general solution can be written as
  $\vec r_5(t) = \mu_4 \vec p_4(t) + (\tau_1,\tau_2,\tau_3)^T$. The curve $\vec r(t)$ in
  Equation~\eqref{eq:ex-FaroukiSir2} is contained in this space with
  \begin{equation*}
    \mu_4 = -20000000,\quad
    \tau_1 = -10736000,\quad
    \tau_2 = 37314400,\quad
    \tau_3 = 15820800.
  \end{equation*}
\end{example}

\begin{example}
  \label{ex:FaroukiSakkalis2019}
  The paper \cite{FaroukiSakkalis2019} is devoted principally to the planar rational PH curves with rational arc-length function. \cite[Example 9]{FaroukiSakkalis2019} however deals with spatial curves and the rational PH curve
  \begin{equation}
    \label{eq:ex-FaroukiSakkalis2019}
    \vec r(t)=\frac{1}{3t(t-1)}
    \begin{pmatrix}
      -4 t^5+22 t^4-18 t^3-10 t^2+34 t\\
      -4 t^5+4 t^4+12t^3+26 t^2-2 t-24\\
      -2 t^5+2 t^4+36 t^3-32 t^2-46 t+18
    \end{pmatrix}
  \end{equation}
  is constructed via the direct integration. It corresponds to the spherical
  motion given by the polynomial
  \begin{equation*}
    \mathcal A = (-t^2+t-1)+(t^3-2 t+2)\qi+(-2 t^3+3 t^2+t-1)\qj+(-t^3+4t^2-2 t+2)\qk.
  \end{equation*}
  Moreover, we have $\alpha = 3t(t-1)$, $n = 1$ and $\beta = t$ or $\beta = t-1$. With $\vec F \coloneqq \mathcal A \qi \mathcal A^\ast$ we have
  \begin{equation*}
    \vec F = (-8\qj + 6\qk)t^0 + (16\qj - 12\qk)t + \cdots
      = (-8\qi - 4\qj + 8\qk)(t-1)^0 + (-16\qi - 8\qj + 16\qk)(t-1)^1 + \cdots
  \end{equation*}
  and can confirm that both choices for $\beta$ satisfy the conditions of
  Lemma~\ref{lem:factors-1}. Once more, we compute solutions via
  \eqref{eq:orthogonality}. For $2 \le \deg \vec b \le 4$ we obtain only trivial
  solutions. For $\deg \vec b = 5$, the solution space is of dimension four. It
  consists of all scaled and translated copies of the curve
  \eqref{eq:ex-FaroukiSakkalis2019}.
\end{example}

Let us finish this section with a remarkable additive decomposition of the
rational PH curve \eqref{eq:ex-FaroukiSakkalis2019}. We have $\vec r(t) = \vec p(t) + \vec q(t)$ where
\begin{equation*}
  \begin{aligned}
    \vec p(t) &=
    \frac{1}{15t}
    \begin{pmatrix}
      -20t^7+108t^6-135t^5-20t^4-30t^3-120t^2-170t\\
      -20t^7+36t^6+30t^5-80t^4-180t^3+360t^2+490t+120\\
      -10t^7+18t^6+90 t^5-190t^4+210t^3+30t^2-130t-90
    \end{pmatrix}\quad\text{and}\\
    \vec q(t) &=
    \frac{1}{15(t-1)}
    \begin{pmatrix}
      20t^7-128t^6+243t^5-135t^4+120t^3\\
      20t^7-56t^6+6t^5+90t^4+120t^3-480t^2+360\\
      10t^7-28t^6-72t^5+270t^4-390t^3+ 360t^2-270
    \end{pmatrix},
  \end{aligned}
\end{equation*}
are both rational PH curves with the same rotation $\mathcal A$ but with simpler
denominator polynomials. We found it computing the solution spaces with the same
quaternion polynomial $\mathcal A$ but with different choices for the denominator
polynomial $\alpha$. We believe that this additive decomposition of a rational
PH curve with denominator $\alpha = 3t(t-1)$ into rational PH curves with
respective denominators $t$ and $t-1$ is only one manifestation of a more
general pattern. Moreover, we expect that the difference between the minimal
degree of a polynomial solution and the minimal degree of a rational PH curve
increases with $n$, and we conjecture that the presence of several linear
factors $\beta$ of $\alpha$ that meet criteria for existence of rational
solutions further increase this difference.

\section{Conclusion}

In this paper we managed to connect the theory of Pythagorean Hodograph curves with the theory of rational framing motions. First benefit of this connection is conceptual. It seems to be appropriate to understand rational spatial PH curves as (trajectories of) framing rational motions (akin curves with rational offsets in the planar case). Second benefit is descriptive and computational. The full set of rational PH curves is described using special motion polynomials in Theorem \ref{th:system-of-equations}. Comparing to the existing results this description contains linear constraints but brings considerable advantages. In particular we were able to understand the cancellation between the numerator and denominator of the PH curves (appearing as a common factor of $\alpha$ and $\vec b$). In Theorem~\ref{th:non-polynomial-solutions} and Remark~\ref{rem:polynomial-solutions} we fully described when this cancellation leads to polynomial curves only.

We are hopeful that our approach will lead to a number of additional results. As a future research we plan to fully describe the basis of the linear system of PH curves with given tangent indicatrix. This opens new possibilities to solve interpolation problems. We have shown how to select infinitely many denominator polynomials $\alpha$ to given $\mathcal A$ such that truly rational solutions curves exist. Generically, the only restriction on $\alpha$ is existence of a factor of multiplicity three. We can choose a finitely generated subspace of this vector space of infinite dimension and add (linear) interpolation constraints. Polynomial PH curves are subsumed in this approach but rationality provides additional degrees of freedom. In this context, the better understanding of possible denominators will help us to avoid singular (cuspidal) interpolants. Finally, we plan to connect our results with the theory of the rotation minimizing frames or PH curves with rational arc length.

\section*{Acknowledgement}

Zbyněk Šír was supported by the grant 20--11473S of the Czech Science Foundation. Bahar Kalkan was supported by the BIDEB 2211-E scholarship programme of The Scientific and Technological Research Council of Turkey.

\bibliographystyle{plainnat}

\end{document}